\documentclass[a4paper,12pt]{article}
\usepackage{amsmath,amsfonts,amsthm}
\usepackage{graphicx}

\def\Prob{{\mathrm{Prob}}}

\newtheorem{proposition}{Proposition}

\begin{document}
\title{The Monty Hall Problem is not \\  a Probability Puzzle\footnote{~ v.4. Revision of v.3 \texttt{arXiv.org:1002.0651}, to appear in \emph{Statistica Neerlandica} }\\{\small (It's a challenge in mathematical modelling)}}
\author{Richard D. Gill\footnote{~Mathematical Institute, University Leiden; during 2010--11 the author is Distinguished Lorentz Fellow at the  Netherlands Institute of Advanced Study, Wassenaar.
\texttt{http://www.math.leidenuniv.nl/$\sim$gill}}}
\date{12 November, 2010}
\maketitle
\begin{abstract}
\noindent \emph{Suppose you're on a game show, and you're given the choice of three doors: Behind one door is a car; behind the others, goats. You pick a door, say No.~1, and the host, who knows what's behind the doors, opens another door, say No.~3, which has a goat. He then says to you, ``Do you want to pick door No.~2?'' Is it to your advantage to switch your choice?} 

The answer is ``yes'' but the literature offers many reasons why this is the correct answer. The present paper argues that the most common reasoning found in introductory statistics texts, depending on making a number of ``obvious'' or ``natural'' assumptions and then computing a conditional probability, is a classical example of solution driven science. The best reason to switch is to be found in von Neumann's minimax theorem from game theory, rather than in Bayes' theorem.
\end{abstract}

\section{Introduction}
\noindent In the above abstract to this paper, I reproduced The Monty Hall Problem, as it was defined by Marilyn vos Savant in her ``Ask Marilyn'' column in \emph{Parade} magazine (p.~16, 9 September 1990). Marilyn's solution to the problem posed to her by a correspondent Craig Whitaker sparked a controversy which brought the Monty Hall Problem to the attention of the whole world. Though MHP probably originated in a pair of short letters to the editor in \emph{The American Statistician} by biostatistician Steve Selvin (1975a,b), from 1990 on it was public property, and has sparked research and controversy in mathematical economics and game theory, quantum information theory, philosophy, psychology, ethology, and other fields, as well as having become a fixed point in the teaching of elementary statistics and probability.

This has resulted in an enormous literature on MHP.  Here I would like to draw attention to the splendid book by Jason Rosenhouse (2009), which has a huge reference list and which discusses the pre-history and the post-history of vos Savant's problem as well as many variants. My other favourite is Rosenthal (2008), one of the few papers where a genuine attempt is made to argue to the layman why MHP has to be solved with conditional probability. Aside from these two references, the English language wikipedia page on the Monty Hall Problem, as well as its discussion page, is a rich though every-changing resource. Much that I write here was learnt from the many editors of those pages, both allies and enemies in the never ending edit-wars which plague it.

The battle among wikipedia editors could be described as a battle between intuitionists versus formalists, or to use other words, between simplists versus conditionalists. The main question which is endlessly discussed is whether simple arguments for switching, which typically show that the \emph{unconditional} probability that the switching gets the car is 2/3, may be considered rigorous and complete solutions of MHP. The opposing view is that vos Savant's question is only properly answered by study of ``the'' \emph{conditional} probability that the switching gets the car, \emph{given} the initial choice of door by the player and door opened by the host. This more sophisticated approach requires making more assumptions, and that leads to the question whether those supplementary conditions are implicitly implied by vos Savant's words. What particularly interests me, however, is that the conditionalists take on a dogmatic stance: their point of view is put forward as a moral imperative. This leads to an impasse, and the clouds of dust thrown up by what seems a religious war conceal what seem to me to be much more interesting, though more subtle, questions.  

My personal opinion on the wikipedia-MHP-wars is that they are fights about the wrong question. Craig Whitaker, through the voice of Marilyn vos Savant, asks for an action, not a probability. I think that game theory gives a more suitable framework in which to represent our ignorance of the mechanics of the set-up (where the car is hidden) and of the mechanics of the host's choice, than subjectivist probability.

Therefore, though Rosenhouse's book is a wonderful resource, I strongly disagree with him, as well as with many other authors, on what the essential Monty Hall problem is (and that is the main point of this paper). Deciding unilaterally (Rosenhouse, 2009) that a certain formulation is \emph{canonical} is asking for schism. Calling a specific version \emph{original} (Rosenthal, 2008) is asking for contradiction. Rosenthal states without any argument at all that additional assumptions are implicitly contained in vos Savant's formulation. Selvin (1975a) did state all those assumptions explicitly but strangely enough did not use all of them. His second paper Sevin (1975b) gave a new solution using all his original assumptions but the author does not seem to notice the difference. At the same time, he quotes with approval a simplist solution of Monty Hall himself, who sees randomness in the choice of the player rather than in the actions of the team who prepare the show in advance, and the quiz-master himself. Vos Savant did not use the full set of assumptions which others find implicit in her question. Her relatively simple explanation of why one should switch seems to satisfy everyone except for the writers of elementary texts in statistics and probability. I have the impression that words like original, canonical, standard, complete are all used to hide the paucity of argument of the writer who needs to make that extra assumption in order to be able to apply the tool which they are particularly fond of, conditional probability. 

One of the most widely cited but possibly least well read papers in MHP studies is Morgan et al.~(1991a), published together with a discussion by Seymann (1991) and a rejoinder Morgan et al.~(1991b).  Morgan et al.~(1991a)  firmly rebuke vos Savant for not solving Whitaker's problem as they consider should be done, namely by conditional probability. They use only the assumption that all doors are initially equally likely to hide the car; this assumption is hidden within their calculations. The paper was written during the peak of public interest and heated emotions about MHP which arose from vos Savant's column.  It actually contains an unfortunate error which was only noticed 20 years later by wikipedia editors Hogbin and Nijdam (2010): if the player puts a non-informative and hence symmetric Bayesian prior on the host's bias in opening a door when he has a choice, it will be equally likely (for the player) that the host will open either door when he has the opportunity.  Morgan et al.~(2010) acknowledge the error and also reproduce part of Craig Whitaker's original letter to Marilyn vos Savant whose wording is even more ambiguous than vos Savant's.

Rosenhouse (2009), Rosenthal (2005, 2008), Morgan et al.~(1991a,b, 2010), and Selvin (1975b) (but not Selvin, 1975a) solve MHP using elementary \emph{conditional} probability. In order to do so they are obliged to add mathematical assumptions to vos Savant's words, without which the conditional probability they are after is not determined. Actually, and I think tellingly, almost no author gives any argument at all why we \emph{must} solve vos Savant's question by computing a conditional probability that the other door hides the car, conditional on which door was first chosen by the player and which opened by the host. 

For whatever reasons, it has become conventional in the elementary statistics literature, where MHP features as an exercise in the chapter on Bayes' theorem in discrete probability, to take it for granted that the car is initially hidden ``at random'', and the host's choice, if he is forced to make one,  is ``at random''  too.  Morgan et al.~(1991a) are notable in only making the first assumption. Many writers also have the player's initial choice ``at random'' too.  ``At random''  is a code phrase for what I would prefer to call \emph{completely} at random. The student is apparently supposed to make these assumptions by default, though sometimes they are listed without motivation as if they are always the right assumptions to make. 

In my opinion, this approach to MHP is an example of \emph{solution driven science}, and hence an example of bad practise in mathematical modelling. Taking for granted that unspecified probability distributions must be uniform or normal, depending on context, is the cause of such disasters as the miscarriage of justice concerning the Dutch nurse Lucia de Berk, or the doping case of the German skater Claudia Pechstein.  Of course, MHP does indeed provide a nice exercise in conditional probability, provided one is willing to fill in gaps without which conditional probability does not help you answer the question whether you should stay or switch. Morgan et al. (1991a)'s original contribution is to notice the minimal condition under which conditional probability does give an unequivocal solution. 

Precisely because of all these issues, MHP presents a beautiful playground for learning the art of mathematical modelling.  For me, MHP is the problem of how to build a bridge from vos Savant's words to a mathematical problem, solve that problem, \emph{and} translate the solution back to the real world. \emph{If}  we use probability as a tool in this enterprise, we are going to have to motivate probabilistic assumptions. We must also \emph{interpret}  probabilistic \emph{conclusions}. Like it or not, the interpretation of probability plays a role, going both directions. 

Real world problems are often brought to a statistician because the person with the question, for some reason or other, thinks the statistician must be able to help them. The client has often already left out some complicating factors, or made some simplifications, which he thinks that the statistician doesn't need to know. The first job of the consulting statistician is to find out what the real question is with which the client is struggling, which may often be very different from the imaginary statistical problem that the client thinks he has. The first job of the statistical consultant is to undo the pre-processing of the question which his client has done for him.

In mathematical model building we must be careful to distinguish the parts of the problem statement which are truly determined by the background real world problem, and parts which represent hidden assumptions of the client who thinks he needs to enlist the statistician's aid and therefore has already pre-processed the original question so as to fit in the client's picture of what a statistician can do. The result of a statistical consultation might often be that the original question posed by the client is reformulated, and the client goes home, happier, with a valuable answer to a more meaningful question than the one he brought to the statistician. Maybe this is the real message which the Monty Hall Problem should be telling us? What if vos Savant's opening words had been ``\emph{Suppose you're \emph{going to be} on a game show tonight. If you make it to the last round, you'll be given the choice of three doors...}''?

\section{The mathematical facts}
In this section, I present some elementary mathematical facts, firstly from probability theory, secondly from game theory. The results are formulated within a mathematical framework which does not make any assumptions restricting the scope of the present discussion. Modelling all the various door choices as random variables does not exclude the case that they are fixed. It also leaves the question completely open how we think of probability: in a frequentist or in a Bayesian sense. I impose only the ``structural'' conditions on the sequence of choices, or moves, which are universally agreed to be implied by vos Savant's story.

\subsection{What probability theory tells us}

I distinguish four consecutive  actions: 

\begin{enumerate}

\item Host: hiding the car before the show behind one of three doors, \textsf{Car}

\item Player: choosing a door, \textsf{P1}

\item Host: revealing a goat by opening a different door, \textsf{Goat}

\item Player: switching or staying, final choice door \textsf{P2}

\end{enumerate}

The doors are conventionally labelled ``1'', ``2'' and ``3'', and we can represent the door numbers specified by the four actions with random variables \textsf{Car, P1, Goat, P2}. Since two doors hide goats and one hides a car and the host knows the location of the car, he can and will open a door different to that chosen by the player and revealing a goat. I allow both the location of the car \textsf{Car} and the initial choice of the player \textsf{P1} to be random, and assume them to be statistically independent. From different modelling points of view, we might want to take either of these two variables to be fixed; the independence assumption is then of course harmless. Given the location of the car and the door chosen by the player, the host opens a different door \textsf{Goat} revealing a goat, according to a probability distribution over the two doors available to him when he has a choice (which includes the case that he follows some fixed selection rule). Then the player makes his choice \textsf{P2}, deterministically or according to a probability distribution if he likes, but in either case only depending on his initial choice and the door opened by the host. Finally we can see whether he goes home with a car or a goat by opening his door and revealing his \textsf{Prize}.

The probabilistic structure of the four actions together with the final result \textsf{Prize} (whether the player goes home with a car or a goat) can be represented in the graphical model or Bayes net shown in Figure 1. This diagram (drawn using the \textsf{gRain} package in the statistical package R) was inspired by Burns and Wieth (2004) who performed psychological experiments to test their hypothesis that people fail MHP because of their inability to internalise the \emph{collider principle}: conditional on a consequence, formerly independent causes become correlated. In this case, the initially statistically independent initial moves \textsf{Car} and \textsf{P1} of host and player are conditionally \emph{dependent} of one another given the door \textsf{Goat} opened by the host.

\begin{figure}
\begin{center}
\includegraphics[width=8cm]{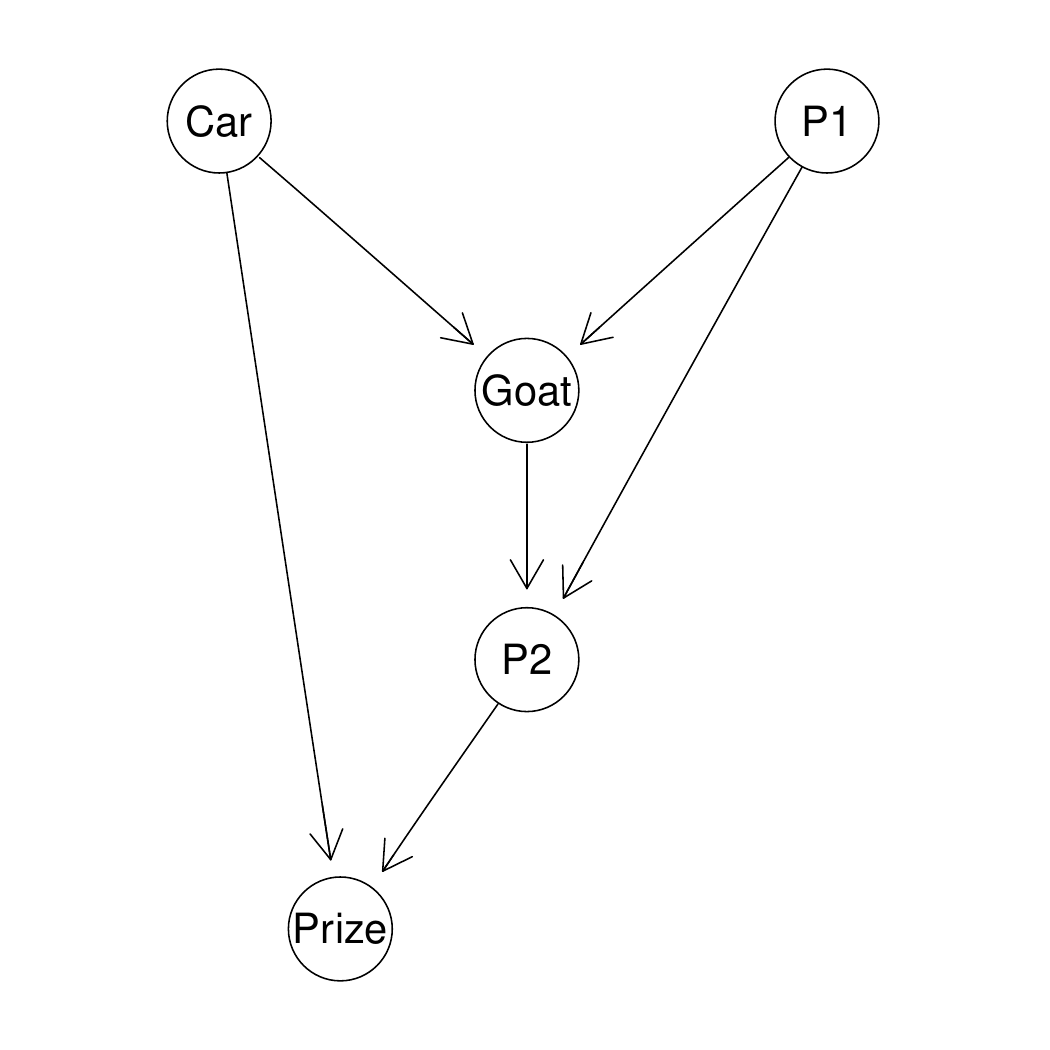}\\[0.2cm]
\caption{{A Graphical Model (Bayes Net) for MHP}}
\end{center}
\end{figure}

I now write down three simple propositions, each making in turn a stronger mathematical assumption, and each in turn giving a better reason for switching.

\begin{proposition} If the player's initial choice has probability $1/3$ to hit the car, then always switching gives the player the car with (unconditional) probability $2/3$ \emph{(Monty Hall, as reported by Selvin, 1975b)}.
\end{proposition}

\begin{proposition} If initially all doors are equally likely to hide the car, then, given the door initially chosen and the door opened,  switching gives the player the car with conditional probability at least $1/2$ \emph{(Morgan et al., 1991a)}. 

Consequently, not only does ``always switching'' give the player the car with unconditional probability $2/3$, but no other strategy gives a higher success chance.
\end{proposition}

\begin{proposition} If initially all doors are equally likely to hide the car and if the host is equally likely to open either door when he has a choice, then,  given the door initially chosen and the door opened,  switching gives the player the car with conditional probability $2/3$, whatever door was initially chosen and which door was opened \emph{(Morgan et al., 1991a,b).}
\end{proposition}

\begin{proof}  ~ 
\\
\noindent \underline{Prop.~1}: This implication is trivial once we observe that a ``switcher'' wins if and only if a ``stayer'' loses.
\\ 
\\
\noindent \underline{Prop.~2}: We use Bayes' theorem in the form
$$\textit{posterior odds equals prior odds times likelihood ratio.}$$
The initial odds that the car is behind doors 1, 2 and 3 are 1:1:1. The posterior odds are therefore proportional to the probabilities that the host opens Door 3 given the player chose Door 1 and the car is behind Door 1, 2 and 3 respectively. These probabilities are $q$, $1$ and $0$ respectively, where 
$$q~=~\Prob(\,\textrm{Host opens Door 3} \,|\, \textrm{Player chose Door 1,  car is behind same}\,).$$
The posterior odds for switching versus staying are therefore $1:q$, so that staying does not have an advantage over switching, whatever $q$ might be. 

Since all doors are initially equally likely to hide the car, the door chosen by the player hides the car with probability 1/3. The unconditional probability that switching gives the car is therefore 2/3. By the law of total probability, this can be expressed as the sum over all six conditions (door chosen by player, door opened by host), of the probability of that condition times the conditional probability that switching gives the car, under the condition. Each of these conditional probabilities is at least 1/2.  The strategy of always switching can't be beaten, since the success probability of any other strategy is obtained from the success probability of always switching by replacing one or more of the conditional probabilities of getting the car by switching by probabilities which are never larger.
\\ 
\\
\noindent \underline{Prop.~3}:  If all doors are equally likely to hide the car then by independence of the initial choice of the player and the location of the car, the probability that the initial choice is correct is $1/3$. Hence the unconditional probability that switching gives the car is $2/3$. If the player's initial choice is uniform and the two probability distributions involved in the host's choices are uniform, the problem is symmetric with respect to the numbering of the three doors. Hence the conditional probabilities we are after in Proposition 3 are all the same, hence by the law of total probability equal to the unconditional probability that switching gives the car, 2/3.

\end{proof}

Proposition 3 also follows from the inspection of the posterior odds computed in the proof of Proposition 2. On taking $q=1/2$, the posterior odds in favour of switching are 2:1 (Morgan et al., 1991a).

In the literature, Proposition 3 is usually proven by explicit computation or tabulation, i.e., by going back to first principles to compute the conditional probability in question. For instance, Morgan et al.~(1991a) also give this direct computation, attributing it to Mosteller's (1965) solution of the Prisoner's dilemma paradox.  It is often offered as an example of Bayes' theorem, but really is just an illustration of conditional probability via its definition. On the other hand, Bayes' theorem in its odds form (which I used to prove Morgan et al.'s Proposition 2) is a genuine \emph{theorem}, and offers to my mind a much more satisying route for those who like to see a computation and at the same time learn an important concept and a powerful tool. To my mathematical mind the most elegant proof of Proposition 3 is the argument by symmetry starting from Proposition 1: the conditional probability is the same as the unconditional since all the conditional probabilities must be the same. I learnt this proof from Boris Tsirelson on wikipedia discussion pages, but it is also to be found in Morgan et al.~(1991b). 

This proof also supplies one reason why the literature is so confused as to what constitutes a solution to MHP. One could apply symmetry at the outset, to argue that we only want an unconditional probability. There is no point in conditioning on anything which we can see in advance is irrelevant to the question at hand.

The pages of wikipedia, as well as a number of papers in the literature, are the scene of a furious controversy mainly as to whether Proposition 1 and a proof thereof, or Proposition 3 and a proof thereof, is a ``complete and correct solution to MHP''.  These two solutions can be called the simple or popular or unconditional solution, and the full or complete or conditional solution respectively. The situation is further complicated by the fact that many supporters of the popular solution do accept all the symmetry (uniformity) conditions of Proposition 3, for a variety of reasons. I will come back to this in the next main section, but first consider a rather different kind of result which can be obtained within exactly the same general framework as before.

\subsection{What game theory tells us}

Let us think of the four actions of the previous subsection as two pairs of moves in a two stage game between the host and the player in which the player wants to get the car, the host wants to keep it. Von Neumann's minimax theorem tells us that there exist a pair of minimax strategies for player and host, and a value of the game, say $p$, having the following defining characteristics. The minimax strategy of the player (minimizes his maximum chance of losing) guarantees him at least probability $p$ of winning, whatever the strategy of the host; the minimax stategy of the host (minimizes his maximum probability of losing) guarantees him at most probability $p$ of losing. If both player and host play their minimax strategy then indeed the player will win with probability $p$.

\begin{proposition}The player's strategy ``initial choice uniformly at random, thereafter switch'' and the host's strategy ``hide the car uniformly at random, open a door uniformly at random when there is a choice'' form the minimax solution of the finite two-person zero-sum game in which the player tries to maximize his probability of getting the car, the host tries to minimize it. The value of the game is $2/3$.
\end{proposition}

\begin{proof}
We must verify two claims. The first is that whatever strategy is used by the host, the player's minimax strategy guarantees the player a success chance of at least 2/3. The second is that whatever strategy is used by the player, the host's minimax strategy prevents the player from achieving a success chance greater than 2/3.

For the first claim notice that if the player chooses a door uniformly at random and thereafter switches, he'll get the car with probability exactly 2/3; that follows from Proposition 1.

For the second, suppose the host hides the car uniformly at random and thereafter opens a door uniformly at random when he has a choice. Making the initial choice of door in any way, and thereafter switching, gives the player success chance 2/3, and by Proposition 2 (or 3, if you prefer) there is no way to improve this.
\end{proof}

Note that I did not use von Neumann's theorem in any way to get this result: I simply made use of the probabilistic results of the previous subsection. 

MHP was brought to the attention of the mathematical economics and game theory community by a paper of Nalebuff (1987), which contains a number of game-theoretic or economics choice puzzles.  He considered MHP as an amusing problem with which to while away a few minutes during a boring seminar. After describing the problem he very briefly reproduced the short solution corresponding to Proposition 1. He enigmatically drops the names of Neumann-Morgenstern and Bayes as he ponders why most people in real life took the wrong decision, but he does not waste any more time on MHP. 

Variants of the MHP in which the host does not always open a door, or where he might be trying to help you, or might be trying to cheat you, lend themselves very well to game theoretic study, see wikipedia or Rosenhouse (2009) for references. 

For present purposes, the important point which I think is brought out by a game theoretic approach is that the player does have two decision moments. The player has control over his initial choice. Vos Savant describes the situation at the moment the player must make up his mind whether to switch or stay, and most, but not all, people will instinctively feel that this is the only important decision moment. But the player earlier had a chance to choose any door he liked. Perhaps he would have been wise to think about what he would do if he did make it to the last round of the show, before setting off to the TV studio. There is no way he can ask the advice of a friendly mathematician as he stands on the stage under the dazzling studio lights while the audience is shouting out conflicting advice.

Van Damme (1995; in Dutch) goes a little deeper into the question of why real human players did not behave rationally on the Monty Hall show; this is one of the main questions studied in the psychology, philosophy, artificial intelligence and animal behaviour literature on MHP.  Since ``rational expectations'' play a fundamental role in modern economic theory, the actual facts of the real world MHP, where players almost never switched doors, bodes ill for the application of economic theory to real world economics. The usual rationale for human rational expectations in economics is that humans learn from mistakes. However, the same person did not get to play several times times in the final round of the Monty Hall show, and apparently no-one kept a tally of what had happened to previous contestants, so learning simply did not take place. Nobody thought there would be a point in learning! Instead, the players used their brains, came to the conclusion that there was no advantage in switching, and mostly stuck to their original choice. At this point they do make a rational choice: there would be a much larger emotional loss to their ego on switching and losing, than on staying and losing. Sticking to your door demonstrates moral fortitude. Switching is feckless and deserves punishment.

Interestingly, pigeons (specifically, the Rock Pigeon, \emph{Columba Livia}, the pigeon which tourists feed in city squares all over the world) are very good at learning how to win the Monty Hall game, see Herbranson and Schroeder (2010). They do not burden their little brains thinking about what to do but just go ahead and choose. There is a lot of variation in their initial decisions whether to switch or stay, and observing the results gives them a chance to learn from the past without thinking at all. Only a very small brain is needed to learn the optimal strategy. And these birds are evolutionarily speaking very succesful indeed.

\section{Which assumptions?}

Propositions 1, 2 and 3 tell us in different ways that switching is a good thing. Notice that the mathematical conditions made are successively stronger and the conclusion drawn is successively stronger too. As the conditions get stronger, the scope of application of the result gets narrower: there are more assumptions to be justified. From a mathematical point of view none of these results are stronger than any of the others: they are all \emph{strictly different}. 

The literature on MHP focusses on variants of Proposition 1, and of Proposition 3.  These correspond to what are called the popular or simple or unconditional solution, and the full or conditional solution to MHP. The full solution is mainly to be found in introductory probability or statistics texts, whereas the simple solution is popular just about everywhere else. The intermediate ``Proposition 2'' is only occasionally mentioned in the literature. The full list of conditions in Proposition 3 is often called, at least in the kind of texts just mentioned, the standard or canonical or original MHP. I will just refer to them as the \emph{conventional supplementary assumptions}.

Regarding the word ``original'', it is a historical fact that Selvin (1975a) gave MHP its name, did this in a statistics journal, and wrote down the conventional full list of uniformity assumptions. He proceeded to compute the \emph{unconditional} probability that switching gives the car by enumeration of \emph{nine} equally likely cases, for which he took both the player's initial choice and the location of the car as uniform random, and of course independent of one another. In his second note, Selvin (1975b), he computed the \emph{conditional} probability using now his full list of assumptions concerning the host's behaviour, and fixing the initial choice of the player, but without noting any conceptual, let alone technical, difference at all with his earlier solution. Of course, the number ``2/3'' is the same. In the same note he quotes with approval from a letter from Monty Hall himself who gave the argument of Proposition 1: switching gives the car with probability 2/3 because the initial choice is right with probability 1/3. We know Monty will open a door revealing a goat. Conditioning on an event of probability one does not change the probability that the initial choice was right. 

Thus Selvin set the seeds for subsequent confusion. Let me call his approach the \emph{practical-minded approach}:

\begin{quote} The \emph{right} answer to MHP is ``2/3''. There are many ways to get to this answer, and I am not too much concerned how you get there. As long as you get the right answer 2/3, we're happy. After all, the whole point of MHP is that the initial instinct of everyone hearing the problem is to say ``since the host could anyway open a door showing a goat, his action doesn't tell me anything. There are still two doors closed so they still are equally likely to hide the car. So the probability that switching would give the car is ``1/2'', so I am not going to switch, thank you.
\end{quote}

Selvin's two papers together gave MHP a firm and more or less standard position in the elementary statistics literature. There is a conventional complete specification of the problem. This enables us to write down a finite sample space and allocate a probability to every single outcome. Usually the player's initial choice is taken, in the light of the other assumptions without loss of generality, as fixed. All randomness is in the actions of the host, or \emph{in our lack of any knowledge} about them. This corresponds to whether the writer has a frequentist or a subjectivist slant, often not explicitly stated, but implicit in verbal hints. The question is not primarily  ``should you switch or stay?'', but ``what is \emph{the} probability, or \emph{your} probability, that switching will give the car?''  Typically, as in Selvin's second, conditional, approach, the player's initial choice is already fixed in the problem statement, the host's two actions are already seen as completely random, whether because we are told they are, objectively, or because we are completely ignorant of how they are made, subjectively. The problem typically features in the chapter which introduces conditional probability and Bayes' theorem in the discrete setting. Thus the problem is posed by a maths teacher who wants the student to learn conditional probability. The problem is further reformulated as ``what is the \emph{conditional} probability that switching will give the car''.

In such a context not much attention is being paid to the meaning of probability. After all, right now we are just busy getting accustomed to its calculus. Most of the examples figure playing cards, dice and balls in urns, and so the probability spaces are usually completely specified (all outcomes can be enumerated) and mostly they are symmetric, all elementary probabilities are equal. The student is either supposed to ``know'', or he is told explicitly, that the car is initially equally likely to be behind any of the three doors. The host is supposed to choose at random (shorthand for uniformly, or completely, at random) when he has a choice. Since these facts are given or supposed to be guessed, the initial choice of the player is irrelevant, and we are indeed always told that the player has already picked Door 1.

Well, if MHP is merely an exercise in conditional probability where the mathematical model is specified in advance by the teacher, then it is clear how we are to proceed. But I prefer to take a step back and to \emph{imagine you are on a game show}. How could we ``know'' these probabilities? This is especially important when one has the task of ``explaining'' MHP to non-mathematicians and to non-statisticians.

This is where philosophy, or if you prefer, metaphysics, raises its head. How can one ``know'' a probability; what does it mean to ``know'' a probability?  

I am not going to answer those questions. But I am going to compare a conventional frequentist view -- probability is out there in the real world -- to a conventional Bayesian view -- probability is in the information which we possess. I hope to do this neutrally, without taking a dogmatic stance myself. It is a fact that many amateur users of probability are instinctive subjectivists, not so many are instinctive frequentists. Let us try to work out where either instinct would take us. An important thing to realise is: Bayesian probabilities in, Bayesian probabilities out; frequentist in, frequentist out. I will also emphasize the difference which comes from seeing randomness in the \emph{host's moves} or in the \emph{player's moves}, and the difference which comes from seeing the question as asking for an \emph{action}, or more passively for a \emph{probability}.

For a subjectivist (Bayesian) the MHP is very simple indeed. We only know what we have been told by vos Savant (Whitaker). The wording ``\emph{say}, Door 1'' and ``\emph{say}, Door 3'' (my italics) emphasize that we know nothing about the behaviour of the host, whether in hiding cars or in opening doors. Knowing nothing, the situation \emph{for us} is indistinguishable from the situation in which we had been told in advance that car hiding and door opening was actually done using perfect fair randomizers (unbiased dice throws or coin tosses). Probability is a representation of our uncertain knowledge about the single instance under consideration. Probability calculus is the unique internally consistent way to manipulate uncertain knowledge. To start off with, since we know nothing, we may as well choose our door initially according to our personal lucky number, so we picked Door 1. Having seen the host open Door 3, we would now be prepared to bet at odds 2:1 that the car is behind Door 2. The new situation is indistinguishable for us from a betting situation with fair odds 2:1 based on a perfect fair randomizer (by which I simply refer to the kind of situation in which subjectivists and objectivists tend to agree on the probabilities, even if they think they mean something quite different). 

Does the Bayesian (a subjectivist) need Bayes' theorem in order to come to his conclusion? I think the answer is \emph{no}. For a subjectivist the door numbers are irrelevant. The problem is unchanged by renumbering of the doors. His beliefs about whether the car would be behind the other door in any of the six situations (door chosen, door opened) would be the same. He has no need to actively compute the conditional probability in order to confirm what he already knows.  He could use Proposition 3 but is only interested in Proposition 1. The symmetry argument of my proof of Proposition 3 is the mathematical expression of his prior knowledge that he may ignore the door numbers and just compute an unconditional probability. Do you notice the symmetry in advance and take advantage of it, or just compute away and notice it afterwards? It doesn't matter. The answer is $2/3$ and it is a conditional and unconditional probability at the same time. 

What is important to realise is that the probability computed by a subjectivist is also a subjective probability. Starting from probabilities which express prior personal expectations, the probability we have derived expresses how our prior personal expectations as to the location of the car should logically be modified on seeing the host open Door 3 and reveal a goat in response to our choice of Door 1. These probabilities say nothing about what we would expect to see if the game was repeated many times. We might well expect to see a systematic bias in the location of the car or the choice of the host. Our uniform prior distributions express the fact that our prior beliefs or prior information about such biases are invariant under permutations of the door numbers.

For a frequentist, MHP is harder -- unless the problem has already been mathematized, and the frequentist has been told that the car is hidden completely at random and the host chooses completely at random (when he has a choice) too. Personally, I don't find this a very realistic scenario. I can think of one semi-realistic scenario, and that is the scenario proposed by Morgan et al.~(1991a). Suppose we have inside information that every week, the car is hidden uniformly at random, in order to make its location totally unpredictable to all players. However Monty's choice of door to open, when he has a choice, is something which goes on in his head at the spur of the moment. In this situation we may as well let our initial choice of door be determined by our lucky number, e.g., Door 1. Proposition 2 tells us that not only is always switching a wise strategy, it tells us that we cannot do better. No need to worry our heads about \emph{what is} the conditional probability. It is never against switching.

There is just one solution which does not require any prior knowledge at all; instead it requires prior action. Taking our cue from the game theoretic solution, we realize that the player has two opportunites to act, not one. We allow ourselves the latitude to reformulate vos Savant's words as ``You are \emph{going to be} on a game show...''. We advise vos Savant, or her correspondent Craig Whitaker, to take fate into his or her own hands. Before the show, pick your lucky number (1, 2 or 3) by a toss of a fair die. When you make it to the final round, choose that door and thereafter switch. By Proposition 1 you'll come home with the car with probability 2/3, and by Proposition 4 that's the best you can hope for.

Both frequentists and subjectivists will agree that you win the car with probability 2/3 in this way. They will likely disagree about whether the conditional probability that you win, given door chosen and door opened, is also 2/3. I think the frequentist will say that he does not know it since he doesn't know anything about the two host actions, while the subjectivist will say that he does know the conditional probability (and it's 2/3) for the very same reason. So what?

\section{Conclusion}

The Monty Hall Problem offers much more to the student than a mindless exercise in conditional probability. It also offers a challenging exercise in mathematical modelling. I notice three important lessons. (1) The more you assume, the more you can conclude, but the more limited are your conclusions. The honest answer is not one mathematical solution but a range of solutions. (2) Whether you are a subjectivist or a frequentist affects the ease with which you might make probabilistic assumptions but simultaneously affects the meaning of the conclusions. (3). Think out of the box. Vos Savant asks for an \emph{action}, not for a \emph{probability}. The player has \emph{two} decision moments during the show, not one.

\section*{References} ~
\vskip -0.5cm
\raggedright 
\frenchspacing 
\parskip 0.2 cm
\leftskip 0.5 cm
\parindent -0.5 cm

Burns, B. D. and Wieth, M. (2004), The collider principle in causal reasoning: why the Monty Hall dilemma is so hard,
 {\it J. Experimental Psychology: General\/} {\bf 133\/} 434--449. {\tt http://www.psych.usyd.edu.au/staff/bburns/Burns\_Wieth04\_man.pdf}

van Damme, E. E. C. (1995), Rationaliteit. {\it Econ. Stat. Berichten} 15--11--1995, 1019.

Herbranson, W. T. and Schroeder, J. (2010), Are birds smarter than mathematicians? Pigeons (Columba livia) perform optimally on a version of the Monty Hall Dilemma.  {\it J. Comp. Psychol.\/}  {\bf 124\/} 1--13. {\tt http://people.whitman.edu/$\sim$herbrawt/HS\_JCP\_2010.pdf}

Hogbin, M. and Nijdam, W. (2010), Letter to the editor. {\it Am. Statist.\/} {\bf 64} 193.

Morgan, J. P., Chaganty, N. R., Dahiya, R. C., and Doviak, M. J. (1991a), Let's make a deal: The player's dilemma. {\it Am. Statist.\/}  {\bf 45\/} 284--287.

Morgan, J. P., Chaganty, N. R., Dahiya, R. C., and Doviak, M. J. (1991b), Rejoinder to Seymann's comment on ``Let's make a deal: the player's dilemma''.  {\it Am. Statist.\/}  {\bf 45\/} 289.

Morgan, J. P., Chaganty, N. R., Dahiya, R. C., and Doviak, M. J.  (2010), Response to Hogbin and Nijdam's letter,  {\it Am. Statist.\/} {\bf 64} 193--194.

Mosteller, F. (1965), \emph{Fifty Challenging Problems in Probability with Solutions}. Dover, New York.

Nalebuff, B. (1987), Puzzles: Choose a curtain, duel-ity, two point conversions, and more, {\it J.  Econ. Perspectives\/} {\bf 1} (2) 157--163.

Rosenhouse, J. (2009), \emph{The Monty Hall Problem}, Oxford University Press.

Rosenthal, J. S. (2005), {\it Struck by Lightning: The Curious World of Probabilities}. Harper Collins, Canada.

Rosenthal, J. S. (2008), Monty Hall, Monty Fall, Monty Crawl, \emph{Math Horizons\/} {\bf 16\/} 5--7. \texttt{http://probability.ca/jeff/writing/montyfall.pdf}

Selvin, S. (1975a), A problem in probability (letter to the editor). {\it Am. Statist.} {\bf 29}  67.

Selvin, S. (1975b), On the Monty Hall problem (letter to the editor). {\it Am. Statist.\/} {\bf  29\/} 134.

Seyman, R. G. (1991), Comment on ``Let's make a deal: the player's dilemma''. {\it Am. Statist.\/} {\bf 64} 287--288.

\end{document}